\documentclass[12pt]{article}

\usepackage{amsmath,amsthm,amsfonts,amssymb,latexsym,amscd}
\usepackage{color}

 \input{prepictex.tex}
 \input{pictex.tex}
 \input{postpictex.tex}

\begin{document}

\newtheorem{theorem}{Theorem}[section]
\newtheorem{corollary}[theorem]{Corollary}
\newtheorem{lemma}[theorem]{Lemma}
\newtheorem{proposition}[theorem]{Proposition}
\newtheorem{conjecture}[theorem]{Conjecture}
\newtheorem{commento}[theorem]{Comment}
\newtheorem{definition}[theorem]{Definition}
\newtheorem{problem}[theorem]{Problem}
\newtheorem{remark}[theorem]{Remark}
\newtheorem{remarks}[theorem]{Remarks}
\newtheorem{example}[theorem]{Example}

\newcommand{\Nb}{{\mathbb{N}}}
\newcommand{\Rb}{{\mathbb{R}}}
\newcommand{\Tb}{{\mathbb{T}}}
\newcommand{\Zb}{{\mathbb{Z}}}
\newcommand{\Cb}{{\mathbb{C}}}

\newcommand{\Ef}{\mathfrak E}
\newcommand{\Gf}{\mathfrak G}
\newcommand{\iGf}{\mathfrak I\mathfrak G}
\newcommand{\Hf}{\mathfrak H}
\newcommand{\Kf}{\mathfrak K}
\newcommand{\Lf}{\mathfrak L}
\newcommand{\Af}{\mathfrak A}
\newcommand{\Bf}{\mathfrak B}
\newcommand{\Uf}{\mathfrak U}
\newcommand{\ff}{\mathfrak f}
\newcommand{\hf}{\mathfrak h}
\newcommand{\Xf}{\mathfrak X}
\newcommand{\Tf}{\mathfrak T}

\def\A{{\mathcal A}}
\def\B{{\mathcal B}}
\def\C{{\mathcal C}}
\def\D{{\mathcal D}}
\def\F{{\mathcal F}}
\def\G{{\mathcal G}}
\def\H{{\mathcal H}}
\def\J{{\mathcal J}}
\def\K{{\mathcal K}}
\def\LL{{\mathcal L}}
\def\N{{\mathcal N}}
\def\M{{\mathcal M}}
\def\N{{\mathcal N}}
\def\OO{{\mathcal O}}
\def\P{{\mathcal P}}
\def\SS{{\mathcal S}}
\def\T{{\mathcal T}}
\def\U{{\mathcal U}}
\def\W{{\mathcal W}}
\def\Z{{\mathcal Z}}

\def\span{\operatorname{span}}
\def\Ad{\operatorname{Ad}}
\def\ad{\operatorname{Ad}}
\def\tr{\operatorname{tr}}
\def\id{\operatorname{id}}
\def\en{\operatorname{End}}
\def\aut{\operatorname{Aut}}
\def\Out{\operatorname{Out}}
\def\per{\operatorname{Per}(X_n)}
\def\la{\langle}
\def\ra{\rangle}
\def\act{\curvearrowright}
\def\om{\overline{\Omega}_2}
\def\bm{\partial\Omega_2}
\def\pk{{\mathcal P}(\D_2^k)}
\def\pd{{\mathcal P}(\D_2)}
\def\sp{\Sigma_2\times\Sigma_2^*}
\def\spk{\Sigma_2^k\times\Sigma_2^k} 
\def\spr{\Sigma_2^r\times\Sigma_2^r} 
\def\one{{\mathbf 1}}

\title{The Action of the Thompson Group $F$ on \\ Infinite Trees}

\author{Jeong Hee Hong\,  
and Wojciech Szyma{\'n}ski\footnote{Supported by the DFF-Research Project 2, `Automorphisms and invariants of operator algebras', 
Nr. 7014--00145B (2017--2022).}}

\date{\small 10 July 2022}

\maketitle

\renewcommand{\sectionmark}[1]{}

\vspace{7mm}
\begin{abstract}
We construct an action of the Thompson group $F$ on a compact space built from 
pairs of infinite, binary rooted trees. The action arises as an $F$-equivariant compactification 
of the action of $F$ by translations on one of its homogeneous spaces, $F/H_2$, corresponding to a certain subgroup $H_2$ of $F$. The representation 
of $F$ on the Hilbert space $\ell^2(F/H_2)$ is faithful on the complex group algebra $\Cb[F]$. 
\end{abstract}

\vfill\noindent {\bf MSC 2020}: 20F65, 22D25

\vspace{3mm}
\newpage


\section{Introduction and preliminaries}

\subsection{Introduction}

Groups naturally manifest themselves through actions on geometric objects. In this spirit, 
increasing attention has been given recently to construction and analysis of various actions 
of the Thompson groups $F$, $T$ and $V$. For example, see Farley's proof of the Haagerup property for the Thompson groups, 
\cite{F}, Jones' construction of representations of groups $F$ and $T$ from planar algebras, \cite{J}, or 
Garncarek's construction of an analog of the principal series for groups $F$ and $T$, \cite{Garn}. 
The main purpose of this paper is to define an action of the Thompson group $F$ on a compact subspace 
of the topological space of inifinite, binary rooted trees. 

Although original definition of the Thompson groups involves mappings of the interval and the circle, see \cite{Hig} and 
\cite{CFP}, we like to think of these groups as sitting inside the unitary group of the Cuntz algebra $\OO_2$, see \cite{B},  
\cite{N}, and section 1.4 below. This point of view has been already exploited by Haagerup and Olesen in their work on the failure of 
inner amenability for groups $T$ and $V$, \cite{HO}. It is this connection between the Thompson groups and the Cuntz algebra 
$\OO_2$ that motivates the construction we produce in the present paper.  

We start with a subgroup $H_2$ of $F$, see section 2 below, and consider the action of $F$ on $F/H_2$ by translations. 
We note that the representation of $F$ on the Hilbert space $\ell^2(F/H_2)$ is faithful on the group algebra $\Cb[F]$.
The coset space $F/H_2$ is naturally identified with collection $\Omega_2$ of certain projections in the diagonal 
maximal abelian subalgebra (abbreviated MASA) $\D_2$ of the 
Cuntz algebra $\OO_2$. Which projections $p$ from $\D_2$ belong to $\Omega_2$ depends only on the value $\tau(p)$, where 
$\tau$ is the canonical trace  on the core uniformly hyperfinite (abbreviated UHF) subalgebra $\F_2$ of $\OO_2$, 
see Theorem \ref{elementsofOmega2} below. 

We use the following strategy. 
Set $\Omega_2$ has a natural embedding into the space of pairs of all subtrees of an infinite, binary rooted 
tree. Then the action of $F$ on $\Omega_2$ extends by continuity to an action of $F$ on the closure $\om$. 
$\Omega_2$ is a discrete and open subspace in $\om$. Its removal leads to an action of $F$ on the topological 
boundary $\bm=\om\setminus\Omega_2$, 
a space homeomorphic to the Cantor space. It turns out that $\bm$ contains a point globally fixed by $F$, somewhat 
analogously to a result from \cite{F08}. 

Our method of construction of the action $F\act\bm$ in the present paper seems quite natural. It is worth pointing out that 
the standard action of $F$ on the unit interval arises in a similar way, as an $F$ equivariant compactification of the action of 
$F$ on its homogeneous space. Indeed, in this case one may take a subgroup of those functions from interval $[0,1]$ to itself which 
fix one interior dyadic rational.  Also, since elements of group $F$ themselves may be viewed as pairs of (finite) trees, 
it is intuitively to be expected that $F$ may act on pairs of infinite trees. It should be noted that infinite trees have already appeared 
in the literature in the context of actions of group $F$, \cite{GS}. 


\subsection{The Thompson group $F$}

The Thompson group $F$ is the group of order preserving piecewise linear homeomorphisms of the closed 
interval $[0,1]$ onto itself that are differentiable except finitely many dyadic rationals and such that all slopes 
are integer powers of $2$. This group has a presentation 
\begin{equation}\label{Finfgenerators}
F=\langle x_0,x_1,\ldots,x_n,\ldots \;|\; x_j x_i=x_i x_{j+1}, \forall i<j\rangle. 
\end{equation}
Remarkably, group $F$ admits a finite presentation as well
\begin{equation}\label{Ffinitegenerators}
F=\langle A,B \;|\; [AB^{-1},A^{-1}BA]=1, \; [AB^{-1}, A^{-2}BA^2]=1\rangle, 
\end{equation}
see \cite{CFP} and \cite{HP}. 
A connection between the two presentations, (\ref{Finfgenerators}) and (\ref{Ffinitegenerators}), 
 is obtained by setting $x_0=A$ and $x_n=A^{-(n-1)}BA^{n-1}$ 
for $n\geq 1$. As shown in \cite{BS}, every element $g$ of $F$ admits a unique normal form
\begin{equation}\label{normalform}
g=x_{j_1}x_{j_2}\cdots x_{j_k} x_{i_l}^{-1}\cdots x_{i_2}^{-1}x_{i_1}^{-1}, 
\end{equation}
where $j_k\neq i_l$, $j_1\leq j_2\leq\ldots\leq j_k$  and $i_1\leq i_2\leq\ldots\leq i_l$, and if both $x_m$ and 
$x_{m}^{-1}$ appear for some $m$, then so does either $x_{m+1}$ or  $x_{m+1}^{-1}$. 
The commutator subgroup $[F,F]$ of $F$ is simple, and it is the smallest proper normal subgroup of $F$. 
An element $f$ of $F$ belongs to $[F,F]$ if and only if the slopes of the graph of $f$ at both $0$ and $1$ are equal to $1$. 
We have $F/[F,F]\cong\Zb^2$. 
For a good introduction to the Thompson groups $F$, $T$ and $V$, we refer the reader to \cite{CFP} and \cite{BB}. 


\subsection{The Cuntz algebra $\OO_n$} 

If $n$ is an integer greater than 1, then the Cuntz algebra $\OO_n$ is a unital, simple,
purely infinite $C^*$-algebra generated by $n$ isometries $S_1, \ldots, S_n$ satisfying
$\sum_{i=1}^n S_i S_i^* = 1$, see \cite{Cun}. We denote by $W_n^k$ the set of $k$-tuples $\mu = \mu_1\ldots\mu_k$
with $\mu_m \in \{1,\ldots,n\}$, and by $W_n$ the union $\cup_{k=0}^\infty W_n^k$,
where $W_n^0 = \{\emptyset\}$. We call elements of $W_n$ multi-indices.
If $\mu \in W_n^k$ then $|\mu| = k$ is the length of $\mu$. 
If $\mu = \mu_1\ldots\mu_k \in W_n$ then $S_\mu = S_{\mu_1} \ldots S_{\mu_k}$
($S_\emptyset = 1$ by convention) is an isometry with range projection $P_\mu=S_\mu S_\mu^*$. 
Every word in $\{S_i, S_i^* \ | \ i = 1,\ldots,n\}$ can be uniquely expressed as
$S_\mu S_\nu^*$, for $\mu, \nu \in W_n$ \cite[Lemma 1.3]{Cun}.
For $\mu,\nu\in W_n$ we write $\mu\prec\nu$ if $\mu$ preceeds $\nu$ in the natural 
lexicographic order. 

We denote by $\F_n^k$ the $C^*$-subalgebra of $\OO_n$ spanned by all words of the form
$S_\mu S_\nu^*$, $\mu, \nu \in W_n^k$, which is isomorphic to the
matrix algebra $M_{n^k}({\mathbb C})$. The norm closure $\F_n$ of
$\cup_{k=0}^\infty \F_n^k$ is the UHF-algebra of type $n^\infty$,
called the core UHF-subalgebra of $\OO_n$, \cite{Cun}. We denote by $\tau$ the 
unique normalized trace on $\F_n$. The core UHF-subalgebra $\F_n$ is the fixed-point 
algebra for the gauge action $\gamma:U(1)\to\aut(\OO_n)$, such that $\gamma_z(S_j)=zS_j$ for 
$z\in U(1)$ and $j=1,\ldots,n$. There exists a faithful conditional expectation $E$ from 
$\OO_n$ onto $\F_n$, given by averaging with respect to the normalized Haar measure.  
Also, for each $m\in\Zb$ we denote by $\OO_n^{(m)}$ the corresponding spectral subspace for 
$\gamma$ in $\OO_n$. Thus, in particular, $\OO_n^{(0)}=\F_n$.  

The $C^*$-subalgebra $\D_n$ of $\OO_n$ generated by projections $P_\mu$, $\mu\in W_n$, is a 
MASA (maximal abelian subalgebra) in $\OO_n$. 
The spectrum of $\D_n$ is naturally identified with $X_n$ --- the full one-sided $n$-shift space. 
That is, $X_n$ is the cartesian product of infinitely many copies of the finite set $\{1,\ldots,n \}$, indexed by the natural numbers $\Nb$.
We denote $\D_n^k:=\D_n\cap\F_n^k$, for $0\leq k<\infty$. 

In what follows, we will consider elements of $\OO_n$ of the form $w=\sum_{(\alpha,\beta)\in\J}
S_\alpha S_\beta^*$, where $\J$ is a finite collection  of pairs $(\alpha,\beta)$ of words $\alpha,\beta$ in $W_n$. 
Of course, such a presentation (if it exists) is not unique. 
We denote $\J_1=\{\alpha\in W_n:\exists (\alpha,\beta)\in\J\}$ and 
$\J_2=\{\beta\in W_n:\exists (\alpha,\beta)\in\J\}$. 
In particular, we consider the group $\SS_n$  of those unitaries in $\OO_n$ which can be written
as finite sums of words, i.e. in the form $u=\sum_{(\alpha,\beta)\in\J}S_\alpha S_\beta^*$ for some finite set $\J\subseteq W_n\times W_n$.  
Such a sum is unitary if and only if $\sum_{\alpha\in\J_1} P_\alpha = 1 = \sum_{\beta\in\J_2} P_\beta$. 
We note that $\SS_n$ is a subgroup of the normalizer of $\D_n$ in $\OO_n$. An element 
$u=\sum_{(\mu,\nu)\in\J}S_\mu S_\nu^*$ of $\SS_n$ may be represented by an oriented bipartite graph  
with (distinct) vertices labeled $\alpha$ and $\beta$, $(\alpha,\beta)\in\J$, and directed edges (arrows) 
from vertex $\beta$ to vertex $\alpha$. 
We will always arrange the vertices in two rows: top labelled by $\beta$'s and bottom labelled by 
$\alpha$'s. For example, $u=S_{21}S_{1}^* + S_{22}S_{21}^* + 
S_{1}S_{22}^*\in\SS_2$ is represented by the following oriented bipartite graph. 
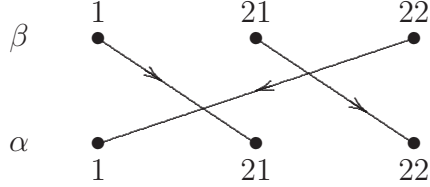
\[ \beginpicture

\setcoordinatesystem units <0.7cm,0.7cm>
\setplotarea x from -5 to 5, y from -2 to 2
\put {$\bullet$} at -3 1
\put {$\bullet$} at -3 -1
\put {$\bullet$} at 0 1
\put {$\bullet$} at 3 1
\put {$\bullet$} at 0 -1
\put {$\bullet$} at 3 -1
\put {$\alpha$} at -4.5 -1
\put {$\beta$} at -4.5 1
\arrow <0.235cm> [0.2,0.6] from -2.1 0.4 to -1.8 0.2
\arrow <0.235cm> [0.2,0.6] from 1.8 -0.2 to 2.1 -0.4
\arrow <0.235cm> [0.2,0.6] from 0.3 0.1 to 0 0
\put{$1$} at -3 1.5
\put{$21$} at 0 1.5
\put{$22$} at 3 1.5
\put{$1$} at -3 -1.5
\put{$21$} at 0 -1.5
\put{$22$} at 3 -1.5
\setlinear
\plot -3 1  0 -1 /
\plot 0 1  3 -1 /
\plot 3 1  -3 -1 /

\endpicture \] 
We will also have each row organized in increasing lexicographic order. Thus, if $\beta_1$ 
precedes $\beta_2$, then two arrows, one from $\beta_1$ to $\alpha_1$ and 
the other from $\beta_2$ to $\alpha_2$, cross if and only if $\alpha_2$ precedes $\alpha_1$.  
If the graph of a unitary  $w=\sum_{(\alpha,\beta)\in\J}S_\alpha S_\beta^*$ contains no 
crossings, then we say $w$ is {\em order preserving}. With help of oriented bipartite graph presentation of unitaries, this order preserving 
propertry will be used in what follows to identify elements of the Thompson group $F$ inside the Cuntz algebra $\OO_2$. 

We note that for an order preserving unitary $w=\sum_{(\alpha,\beta)\in\J}S_\alpha S_\beta^*$, its 
inverse $w^*=\sum_{(\alpha,\beta)\in\J}S_\beta S_\alpha^*$ is order preserving as well. Moreover, let 
$u=\sum_{(\mu,\nu)\in\K}S_\mu S_\nu^*$ be another order preserving unitary. We can choose such 
representations of $u$ and $w$ that the lengths of words $\nu$ and $\alpha$ are all the same, that is 
$|\nu|=|\alpha|=k$ for some fixed $k$. Since $1=\sum_{|\kappa|=r}S_\kappa S_\kappa^*$ for each $r$, 
this can be achieved replacing terms $S_\alpha S_\beta^*$ by 
$\sum_{|\kappa|=k-|\alpha|}S_\alpha S_\kappa S_\kappa^*S_\beta^*$ and terms $S_\mu S_\nu^*$ by 
$\sum_{|\kappa|=k-|\nu|}S_\mu S_\kappa S_\kappa^*S_\nu^*$. Then we see that 
$$
uw = \sum_{(\mu,\nu)\in\K,(\alpha,\beta)\in\J, } S_\mu S_\nu^*S_\alpha S_\beta^* 
= \sum_{(\mu,\nu)\in\K, (\alpha,\beta)\in\J, \,\nu=\alpha} S_\mu S_\beta^*
$$
is order preserving as well. For example, consider $u=S_{11}S_1^* + S_{12}S_{21}^* + S_2S_{22}^*$ and 
$w=S_{111}S_{11}^* + S_{112}S_{121}^* + S_{12}S_{122}^* + S_{211}S_{21}^* + S_{212}S_{221}^* + S_{22}S_{222}^*$. 
We represent them as 
$$ \begin{aligned}
u & =  S_{1111}S_{111}^* +  S_{1112}S_{112}^* +  S_{1121}S_{121}^* +  S_{1122}S_{122}^* + 
S_{121}S_{211}^* + S_{122}S_{212}^* + S_{21}S_{221}^* \\
 &  + S_{22}S_{222}^* \\ 
w & =  S_{111}S_{11}^* + S_{112}S_{121}^* + S_{121}S_{1221}^* + S_{122}S_{1222}^* + 
  S_{211}S_{21}^* + S_{212}S_{221}^* + S_{221}S_{2221}^* \\ 
 & + S_{222}S_{2222}^*, 
\end{aligned} $$
and then see that 
$$ \begin{aligned}
uw & =  S_{1111}S_{11}^* + S_{1112}S_{121}^* + S_{1121}S_{1221}^* + S_{1122}S_{1222}^* + S_{121}S_{21}^*   
S_{122}S_{221}^* + S_{21}S_{2221}^* + \\
 & + S_{22}S_{2222}^*.
\end{aligned} $$


\subsection{The representation of the Thompson groups in $\U(\OO_2)$}

As shown in \cite{B} and \cite{N}, Thompson's group $F$ has a natural faithful representation in the unitary 
group of $\OO_2$ by those  unitaries $u=\sum_{(\alpha,\beta)\in\J}S_\alpha S_\beta^*$ 
in $\SS_2$ that the association $\J_1\ni\alpha\mapsto\beta\in\J_2$ (with $(\alpha,\beta)\in\J)$ 
respects the lexicographic order on $W_2$. Thus, a unitary $u\in\SS_2$ belongs to $F$ if and only if 
it is order preserving. In particular, we have 
\[ \begin{aligned}
x_0 & = S_{11}S_1^* + S_{12}S_{21}^* + S_2S_{22}^*, \\
x_k & = 1-S_2^kS_2^{*k} + S_2^kx_0S_2^{*k}, \;\;\; \text{for } k\geq 1. 
\end{aligned} \]
The two generators $x_0$ and $x_1$ are represented by the following order preserving oriented bipartite graphs:
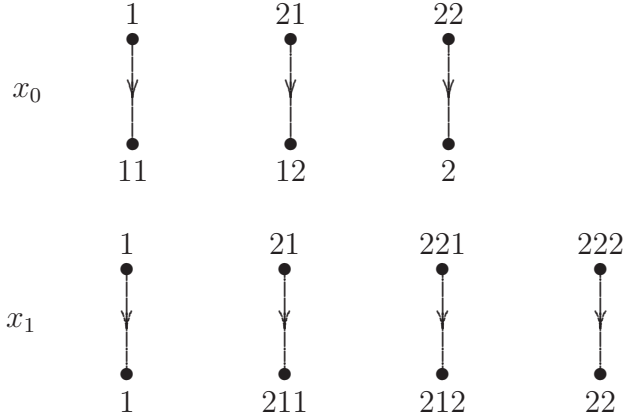
\[ \beginpicture

\setcoordinatesystem units <0.7cm,0.7cm>
\setplotarea x from -5 to 5, y from -2 to 2
\put {$x_0$} at -5 0
\put {$\bullet$} at -3 1
\put {$\bullet$} at -3 -1
\put {$\bullet$} at 0 1
\put {$\bullet$} at 3 1
\put {$\bullet$} at 0 -1
\put {$\bullet$} at 3 -1
\arrow <0.235cm> [0.2,0.6] from -3 0.2 to -3 -0.1
\arrow <0.235cm> [0.2,0.6] from 0 0.2 to 0 -0.1
\arrow <0.235cm> [0.2,0.6] from 3 0.2 to 3 -0.1
\put{$1$} at -3 1.5
\put{$21$} at 0 1.5
\put{$22$} at 3 1.5
\put{$11$} at -3 -1.5
\put{$12$} at 0 -1.5
\put{$2$} at 3 -1.5
\setlinear
\plot -3 1  -3 -1 /
\plot 0 1  0 -1 /
\plot 3 1  3 -1 /

\endpicture \] 
\[ \beginpicture

\setcoordinatesystem units <0.7cm,0.7cm>
\setplotarea x from -6.5 to 6, y from -2 to 2
\put {$x_1$} at -5 0
\put {$\bullet$} at -3 1
\put {$\bullet$} at -3 -1
\put {$\bullet$} at 0 1
\put {$\bullet$} at 3 1
\put {$\bullet$} at 0 -1
\put {$\bullet$} at 3 -1
\put {$\bullet$} at 6 1
\put {$\bullet$} at 6 -1
\arrow <0.235cm> [0.2,0.6] from -3 0.2 to -3 -0.1
\arrow <0.235cm> [0.2,0.6] from 0 0.2 to 0 -0.1
\arrow <0.235cm> [0.2,0.6] from 3 0.2 to 3 -0.1
\arrow <0.235cm> [0.2,0.6] from 6 0.2 to 6 -0.1
\put{$1$} at -3 1.5
\put{$21$} at 0 1.5
\put{$221$} at 3 1.5
\put{$1$} at -3 -1.5
\put{$211$} at 0 -1.5
\put{$212$} at 3 -1.5
\put{$222$} at 6 1.5
\put{$22$} at 6 -1.5
\setlinear
\plot -3 1  -3 -1 /
\plot 0 1  0 -1 /
\plot 3 1  3 -1 /
\plot 6 1  6 -1 /

\endpicture \] 
The subgroup of $\SS_2$ generated by $F$ and $S_{22}S_1^* + S_1S_{21}^* + S_{21}S_{22}^*$ 
is isomorphic to the Thompson group $T$, and consists of those unitaries 
$u=\sum_{(\alpha,\beta)\in\J}S_\alpha S_\beta^*$ 
in $\SS_2$ that the association $\J_1\ni\alpha\mapsto\beta\in\J_2$ (with $(\alpha,\beta)\in\J)$ 
respects the lexicographic order on $W_2$ up to a cyclic permutation. 
Finally, group $\SS_2$ itself is isomorphic to the Thompson group $V$. 


\section{The action of group $F$ on $F/H_2$}

We  consider a subgroup $H_2$ of $F$, defined as 
\begin{equation}\label{hn}
H_2:=\left\{\sum_{(\alpha,\beta)\in\J}S_\alpha S_\beta^* \in F \mid \forall (\alpha,\beta)\in\J, \;  
|\alpha|-|\beta| \in 2\Zb\right\}.
\end{equation}
We are interested in the action of group $F$ on $F/H_2$ by left translations. 

Let $f=\sum_{(\alpha,\beta)\in\J}S_\alpha S_\beta^*$ be in $F$. We set 
$$ \begin{aligned}
f_0 & := \sum_{|\alpha|-|\beta|\in2\Nb}S_\alpha S_\beta^*, \\
f_1 & := \sum_{|\alpha|-|\beta|\not\in2\Nb}S_\alpha S_\beta^*, 
\end{aligned} $$
with the summations above involving only $(\alpha,\beta)\in\J$. 
We denote by $[f]$ the coset in $F/H_2$ determined by $f\in F$. We have $[g]=[f]$ if and only if 
$f^{-1}g\in H_2$. Since 
$$
f^{-1}g = (f_0^* + f_1^*)(g_0 + g_1) = (f_0^*g_1 + f_1^*g_0) + (f_0^*g_0 + f_1^*g_1), 
$$
$f^{-1}g\in H_2$ iff $f_0^*g_1 + f_1^*g_0 = 0$, and clearly this is equivalent to $g_0g_0^*=f_0f_0^*$. That is, 
projection $f_0f_0^*$ determines uniquely the coset of $f$. Thus, denoting
\begin{equation}\label{omega}
\Omega_2:=\{f_0f_0^* \in\D_2 \mid f\in F\},
\end{equation}
we see that the coset space $F/H_2$ is naturally identified with $\Omega_2$. Furthermore, 
let $f,g\in F$ and denote $p:=g_og_0^*$. Then $(fg)_0 = f_0g_0 + f_1g_1$, and thus 
$(fg)_0(fg)_0^* = f_0 p f_0^* + f_1 (1-p) f_1^*$. Hence, under the above identification of 
$F/H_2$ with $\Omega_2$, the action of $F$ on $F/H_2$ corresponds to the following action of $F$ on $\Omega_2$:
\begin{equation}\label{actionFH2}
f\cdot p = f_0 p f_0^* + f_1 (1-p) f_1^*,  
\end{equation}
for $f\in F$ and $p\in\Omega_2$. In particular, we have 
$$ \begin{aligned}
x_0\cdot p & = (P_{11} - S_{11}(S_1^*pS_1)S_{11}^*) + S_{12}(S_{21}^*pS_{21})S_{12}^* 
                   + (P_2 - S_2(S_{22}^*pS_{22})S_2^*), \\
x_0^{-1}\cdot p & = (P_1 - S_1(S_{11}^*pS_{11})S_1^*) + S_{21}(S_{12}^*pS_{12})S_{21}^* 
                   + (P_{22} - S_{22}(S_2^*pS_2)S_{22}^*), \\
x_1\cdot p & = S_1(S_1^*pS_1)S_1^* + (P_{211} - S_{211}(S_{21}^*pS_{21})S_{211}^*) 
                   + S_{212}(S_{221}^*pS_{221})S_{212}^* \\
                & + (P_{22} - S_{22}(S_{222}^*pS_{222})S_{22}^*), \\
x_1^{-1}\cdot p & = S_1(S_1^*pS_1)S_1^* + (P_{21} - S_{21}(S_{211}^*pS_{211})S_{21}^*)  
                   + S_{221}(S_{212}^*pS_{212})S_{221}^* \\ 
                & + (P_{222} - S_{222}(S_{22}^*pS_{22})S_{222}^*). 
\end{aligned} $$  

The action of group $F$ on $\Omega_2$ by (\ref{actionFH2}) 
gives rise to a unitary representation $\pi_2$ of group $F$ on the Hilbert space $\ell^2(\Omega_2)$. 

\begin{proposition}\label{algebraicfaithfulness}
Representation $\pi_2$ is faithful on the group algebra $\Cb[F]$. 
\end{proposition}
\begin{proof}
At first, consider an element $f$ of $F$ and a projection $p$ in $\D_2$ such that $fp\neq p$. Then 
there exists a projection $q\leq p$ in $\D_2$ such that $qfq=0$. We define an element $g\in F$, 
as follows. If $fq$ belongs to the union of subspaces $\OO_2^{(m)}$ with all $m$ even, then we take $g$ 
to be an arbitrary element in $F\setminus H_2$ such that $g(1-q)=(1-q)g=1-q$. Otherwise we put $g=1$. 
In both cases, we have $g^{-1}fg\not\in H_2$. 

Now, we claim that for any elements $f_1,\ldots,f_k$ of $F$, all different from the identity, there is 
a $g\in F$ such that $g^{-1}f_j g\not\in H_2$ for all $j=1,\ldots,k$. Indeed, it is easily seen that there exist 
mutually orthogonal projections $q_1,\ldots,q_k$ in $\D_2$ such that $q_jf_jq_j=0$ and 
$f_jq_jf_j^*(\sum_{i\neq j}q_i)=0$ for all $j=1,\ldots,k$. For each $j$, let $g_j$ be an element of $F$ 
such that $g_j(1-q_j)=(1-q_j)g_j=1-q_j$ and $g_j^{-1}f_jg_j\not\in H_2$. Put $g:=g_1g_2\cdots g_k$. 
Then $g^{-1}f_j g\not\in H_2$ for all $j=1,\ldots,k$, as claimed. 

Now, let $f_1,\ldots,f_k$ be distinct elements of $F$. By the above argument, there exists a $g\in F$ such that 
$g^{-1}f_i^{-1}f_jg\not\in H_2$ for all $i\neq j$. Let $p=g\cdot 1$, an element of $\Omega_2$. 
Then $f_1\cdot p,\ldots,f_k\cdot p$ are distinct elements of $\Omega_2$. This immediately implies that 
$\pi_2(f_1),\ldots,\pi_2(f_k)$ are linearly independent, and hence representation $\pi$ is faithful. 
\end{proof}

\begin{proposition}\label{f0nonzero}
For each $f\in F$, we have $f_0\neq0$. 
\end{proposition}
\begin{proof}
Suppose, by way of contradiction, that $f\in F$ and $f_0=0$. Write
\begin{equation}\label{fzero}
f= \sum_{|\alpha|>|\beta|}S_\alpha S_\beta^* + \sum_{|\mu|<|\nu|}S_\mu S_\nu^*. 
\end{equation}
We can arrange such a representation of $f$ as in (\ref{fzero}) so that for some $k\in\Nb$ we have 
$|\alpha|=k=|\nu|$ for all $(\alpha,\beta)$ and all $(\mu,\nu)$. Now, denote 
$$ \begin{aligned}
r_j & = \#\{(\alpha,\beta)\in\J \mid |\alpha|-|\beta|=2j+1\}, \;\; j=0,1,\ldots,R, \\
l_j & = \#\{(\mu,\nu) \in\J\mid |\nu|-|\mu|=2j+1\}, \;\; j=0,1,\ldots,L. 
\end{aligned} $$
We have $1=\sum P_\alpha + \sum P_\mu = \sum P_\beta + \sum P_\nu$. Applying the rescaled trace 
$2^k\tau$, we get
$$ 2^k  = \sum_{i=0}^R r_i + \sum_{j=0}^L l_j 2^{2j+1} = \sum_{i=0}^R r_i 2^{2i+1} + \sum_{j=0}^L l_j . $$
Adding these together, we obtain the identity 
\begin{equation}\label{2kequality}
2^{k+1} = \sum_{i=0}^R r_i(2^{2i+1}+1) + \sum_{j=0}^L l_j(2^{2j+1}+1). 
\end{equation}
This, however, is a contradiction, as the RHS of (\ref{2kequality}) is divisible by $3$, while the LHS is not. 
\end{proof}

Now, we are ready to give a detailed description of space $\Omega_2$, defined in (\ref{omega}). 

\begin{theorem}\label{elementsofOmega2}
$\Omega_2$ consists of all those projections $p\in\D_2$ that $\tau(p)=k/2^{2m+1}$, where $m\in\Nb$ and 
$k\in\Nb$ are such that $0< k \leq 2^{2m+1}$ and $k-2$ is divisible by $3$.  
\end{theorem}
\begin{proof}
At first, we show that every projection $p\in\Omega_2$ satisfies the condition of the theorem. Let $f\in F$ be written 
as a word $w_1\cdots w_r$, with each $w_j\in\{x_0,x_0^{-1},x_1,x_1^{-1}\}$. We will show that projection 
$f_0f_0^*$ is of the required form, proceeding by induction on $r$. The case $r=0$ corresponding to $f=f_0
f_0^*=1$ is clear. For the inductive step, suppose that $p=(w_2\cdots w_r)_0(w_2\cdots w_r)_0^*$ has the 
required form with integer $m$. We have $f_0f_0^*=w_1\cdot p$. If $w_1=x_0$, then 
$$ \tau(w_1\cdot p) =  \frac{1}{2}\left( \frac{1}{2} - \tau(pP_1) \right) + \tau(pP_{21}) 
   + 2\left( \frac{1}{4} - \tau(pP_{22}) \right). $$
Thus we have 
$$ 2^{2m+1}(\tau(w_1\cdot p) - \tau(p)) = 2^{2m+1}3\left( \frac{1}{4} - \frac{\tau(pP_1)}{2} 
   - \tau(pP_{22}) \right). $$
Since the RHS of this identity is divisible by $3$, so is the LHS and the inductive step follows in this case. 
The remaining three cases when $w_1$ equals $x_0^{-1}$, $x_1$ or $x_1^{-1}$ are handled similarly. 

Now, we show that all projections satisfying the condition of the theorem appear in $\Omega_2$. So let $p=\sum_{\alpha\in J}P_\alpha$, 
where $J$ is a subset of $W_2^{2m+1}$ whose number of elements equals 
$2$ modulo $3$. Assume first that $J$ has at least $5$ elements. Let $\alpha_1\prec\alpha_2\prec\alpha_3$ be 
elements of $J$, and let $\mu_1,\ldots,\mu_k$, $\nu_1,\ldots,\nu_l$ 
be such elements of $W_2^{2m+1}$ that $[\alpha_1,\mu_1,\ldots,\mu_k,\alpha_2,\nu_1,\ldots,\nu_l,
\alpha_3]$ is a $\prec$ interval in $W_2^{2m+1}$. Then there exist order preserving partial isometries 
$u_1,\ldots,u_5$ such that: 
$$ \begin{aligned}
(i & ) \;\;\;u_1^*u_1=P_{\alpha_1}, \;\; u_1u_1^*=P_{\alpha_1 1}, \;\; u_1\in\OO_2^{(1)}, \\
(ii & ) \;\;\; u_2^*u_2=\sum_{j=1}^k P_{\mu_j}, \;\; u_2u_2^*=P_{\alpha_1 2} + \sum_{j=1}^{k-1}
P_{\mu_j} + P_{\mu_k 1},  \;\; u_2\in\OO_2^{(0)}, \\
(iii & ) \;\;\; u_3^*u_3 = P_{\alpha_2}, \;\; u_3u_3^* = P_{\mu_k 2}, \;\; u_3\in\OO_2^{(1)}, \\
( iv & ) \;\;\; u_4^*u_4 = \sum_{i=1}^l P_{\nu_i}, \;\; u_4u_4^* = P_{\alpha_2} + \sum_{i=1}^{l-1}
P_{\nu_i}, \;\; u_4\in\OO_2^{(0)}, \\
( v & ) \;\;\; u_5^*u_5 = P_{\alpha_3}, \;\; u_5u_5^* = P_{\nu_l} + P_{\alpha_3}, \;\;\;u_5\in\OO_2^{(-1)}, 
\end{aligned} $$
with the obvious modifications if $k=0$ or $l=0$. Then $f=1-\sum_{j=1}^5 u_j^*u_j + \sum_{j=1}^5 u_j$ 
is an element of group $F$ such that $f\cdot p = \sum_{\alpha\in J\setminus\{\alpha_1,\alpha_2,\alpha_3\}}
P_\alpha$. Thus, starting with $p=1$ and applying this procedure recursively, we can realize every projection 
$p$ as above in the form $f\cdot 1$ for a suitable $f\in F$. 
\end{proof}


\section{The $F$-equivariant compactification of $\Omega_2$}

Let $\Tf_2$ be an infinite, binary rooted tree, as on the picture below. 

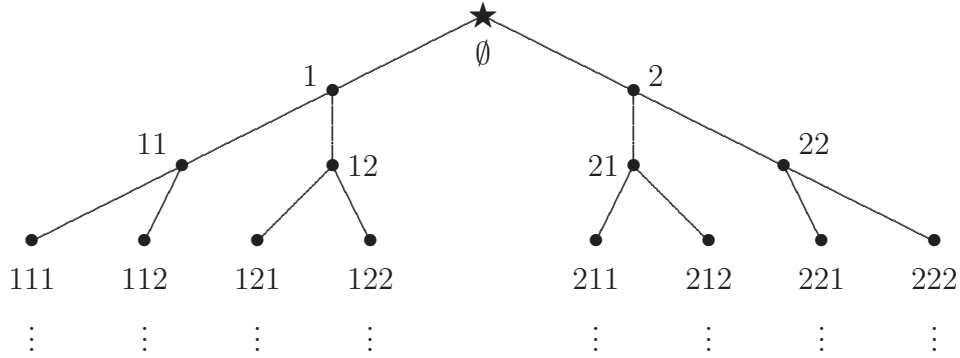
\[ \beginpicture

\setcoordinatesystem units <1cm,1cm>
\setplotarea x from -6 to 6, y from -5 to 0
\put {$\bigstar$} at 0 0
\put {$\bullet$} at -2 -1
\put {$\bullet$} at 2 -1
\put {$\bullet$} at -4 -2
\put {$\bullet$} at 4 -2
\put {$\bullet$} at -2 -2
\put {$\bullet$} at 2 -2
\put {$\bullet$} at -6 -3
\put {$\bullet$} at 6 -3
\put {$\bullet$} at -4.5 -3
\put {$\bullet$} at 4.5 -3
\put {$\bullet$} at -1.5 -3
\put {$\bullet$} at 1.5 -3
\put {$\bullet$} at -3 -3
\put {$\bullet$} at 3 -3
\put {$\vdots$} at -6 -4.2
\put {$\vdots$} at 6 -4.2
\put {$\vdots$} at -4.5 -4.2
\put {$\vdots$} at 4.5 -4.2
\put {$\vdots$} at 3 -4.2
\put {$\vdots$} at -3 -4.2
\put {$\vdots$} at 1.5 -4.2
\put {$\vdots$} at -1,5 -4.2
\put {$\emptyset$} at 0 -0.5
\put {$111$} at -6 -3.5
\put {$112$} at -4.5 -3.5
\put {$121$} at -3 -3.5
\put {$122$} at -1.5 -3.5
\put {$211$} at 1.5 -3.5
\put {$212$} at 3 -3.5
\put {$221$} at 4.5 -3.5
\put {$222$} at 6 -3.5
\put {$11$} at -4.4 -1.7
\put {$22$} at 4.4 -1.7
\put {$1$} at -2.3 -0.8
\put {$2$} at 2.3 -0.8
\put {$12$} at -1.6 -2
\put {$21$} at 1.6 -2
\setlinear
\plot -6 -3  0 0  6 -3 /
\plot -4.5 -3  -4 -2 /
\plot 4.5 -3  4 -2 /
\plot -1.5 -3  -2 -2  -2 -1 /
\plot 1.5 -3  2 -2  2 -1 /
\plot -3 -3  -2 -2 /
\plot 3 -3  2 -2 /

\endpicture \] 
Vertices of $\Tf_2$ are labeled by words from $W_2$. Note that the length of a word equals the distance of the 
corresponding vertex from the root. 

We denote by $\Sigma_2$ the collection of all subtrees of $\Tf_2$ containing the root $\bigstar$ and not 
containing any leaves. That is, if $\omega\in\Sigma_2$ and $\alpha\in\omega$ (the vertex labeled $\alpha$ 
belongs to $\omega$) then at least one of the two successors of $\alpha$ must belong to $\omega$. Let $\Sigma_2^*$ 
denote the set $\Sigma_2$ enlarged by the empty tree. 
There is a natural product topology on $\Sigma_2^*$, making it a compact, metrizable space. 
This space was considered earlier, for example in \cite{Mil}. A sequence 
$(\omega_n)$ of subtrees of $\Tf_2$ converges to $\omega$ if and only if for each $\alpha\in W_2$ we have $\alpha\in\omega\Leftrightarrow 
\alpha\in\omega_n$ for sufficiently large $n$. 

Let $\P(\D_2)$ be the set of all projections in $\D_2$. 
There is a natural embedding of $\P(\D_2)$ into $\Sigma_2^*$. Namely, if $q\in\P(\D_2)$ 
then the corresponding tree (which for simplicity sake will be also denoted $q$) consists of all such vertices 
$\alpha$ that $qP_\alpha\neq0$. Now, we consider $\Omega_2$ as a subset of $\Sigma_2\times\Sigma_2^*$, so that 
\begin{equation}\label{diagonalembed}
\Omega_2\ni q \leftrightarrow (q,1-q)\in\Sigma_2\times\Sigma_2^*. 
\end{equation}
If $q\in\Omega_2$ and  $f\in F$ then we have a nice symmetric form of the action
\begin{equation}\label{actonclosure}
f\cdot(q,1-q) = (f_0 q f_0^* + f_1(1-q)f_1^*,\; f_1 q f_1^* + f_0(1-q)f_0^*). 
\end{equation}

\begin{proposition}\label{continuity}
Let $(q_n)$ be a sequence of elements of $\Omega_2$ that is convergent  in $\sp$. Then for each $f\in F$ 
the sequence $(f\cdot q_n)$ is also convergent in $\sp$.  
\end{proposition}
\begin{proof}
Suppose the sequence $(q_n)$ is convergent in $\sp$.  This is equivalent to saying that for each $\alpha\in W_2$, 
\begin{description}
\item{(i)} if $q_nP_\alpha\neq 0$ for infinitely many $n$ then $q_nP_\alpha\neq 0$ for all but finitely many $n$, and 
\item{(ii)}  if $(1-q_n)P_\alpha\neq 0$ for infinitely many $n$ then $(1-q_n)P_\alpha\neq 0$ for all but finitely many $n$. 
\end{description}
To prove the proposition, we must show that conditions (i) and (ii) hold with $q_n$ replaced by $f\cdot q_n$. We begin 
with (i). Suppose $f\cdot q_nP_\alpha\neq 0$ for infinitely many $n$. Then either $f_0q_nf_0^*P_\alpha\neq 0$ for 
infinitely many $n$ or $f_1(1-q_n)f_1^*P_\alpha\neq 0$ for infinitely many $n$. Suppose the former holds. Then also 
$q_n(f_0^*f_0)(f_0^*P_\alpha f_0)\neq 0$ holds. Let $(f_0^*f_0)(f_0^*P_\alpha f_0)=\sum_\beta P_\beta$. There is a $\beta$ 
such that $q_n P_\beta\neq0$ for infinitely many $n$. But then $q_n P_\beta\neq0$ for all but finitely many $n$, by 
assumption. Hence condition (i) holds for $f\cdot q_n$ in this case. The latter case is verified by an analogous argument, 
and so is condition (ii) for $f\cdot q_n$. 
\end{proof}

Let $\om$ denote the closure of $\Omega_2$ inside $\sp$. 

\begin{definition}\rm 
By virtue of Proposition \ref{continuity}, the action of $F$ on $\Omega_2$ has a unique extension to an action of $F$ on 
$\om$ by homeomorphisms.  
\end{definition}

Indeed, let $\omega\in\om$ and take a sequence $q_n\in\Omega_2$ converging to $\omega$. For an $f\in F$ we put 
$f\cdot\omega:=\lim f\cdot q_n$, knowing that this limit exists by Proposition \ref{continuity}. If $p_n\in\Omega_2$ is another 
sequence converging to $\omega$ then $\lim f\cdot q_n=\lim f\cdot p_n$, since the sequence $h_{2n}:=q_n$, $h_{2n+1}:=p_n$ 
also converges to $\omega$ and thus sequence $f\cdot h_n$ has a limit, again by Proposition \ref{continuity}. This yields a well-defined 
mapping $\om\ni\omega\mapsto f\cdot\omega\in\om$. 

This mapping is continuous. Indeed, let $\omega_n\in\om$ be a sequence converging 
to $\omega$. For each $n$, let $q_m^n\in\Omega_2$ be a sequence converging to $\omega_n$, and let $q_m\in\Omega_2$ be 
a sequence converging to $\omega$.  We have $f\cdot q_m^n\underset{m}{\rightarrow}f\cdot\omega_n$ and 
$f\cdot q_m\underset{m}{\rightarrow}f\cdot\omega$.  Now, equip $\om$ with a metric $d$ compatible with its topology, and let $k_n$ 
be an increasing sequence of natural numbers such that $d(q_{k_n}^n,\omega_n)\to 0$ and $d(f\cdot q_{k_n}^n,f\cdot \omega_n)\to 0$. 
Then also $q_{k_n}^n\to \omega$, and thus $f\cdot q_{k_n}^n\to f\cdot \omega$. We conclude that $f\cdot\omega_n \to f\cdot \omega$. 

Now, for each $f\in F$ we have $f\cdot(f^{-1}\cdot q) = q = f^{-1}\cdot(f\cdot q)$ for all $q\in\Omega_2$. By continuity, we get 
that each mapping $\om\ni\omega\mapsto f\cdot\omega\in\om$ is a homeomorphism with inverse given by the action of $f^{-1}$. Finally, the action 
of $F$ on $\om$ is associative by continuity, since it is associative on the dense subspace $\Omega_2$.

\begin{remark}\label{Lipschitzcompletion}\rm
Obviously, construction of an equivariant compactification depends on a careful choice of topology on $\Omega_2$, 
and some natural topologies do not work properly. 
For example, consider a metric $d_\tau$ on $\Omega_2$, defined as $d_\tau(p,q) := \tau(|p-q|)$.  
For an element $f=\sum_{(\mu,\nu)\in\J}S_\mu S_\nu^*$ of $F$, let 
\begin{equation}\label{height}
\ell(f)=\max\{||\alpha|-|\beta|| \mid (\alpha,\beta)\in\J\}. 
\end{equation}
Then for each $f\in F$ and any two $p,q\in\Omega_2$ we have 
$$ d_\tau(f\cdot p,f\cdot q) = \tau(f_0|p-q|f_0^*) + \tau(f_1|p-q|f_1^*) \leq 2^{\ell(f)+1}
   \tau(|p-q|) = 2^{\ell(f)+1}d_\tau(p,q). $$
Thus, the map $\Omega_2\ni p\mapsto f\cdot p\in\Omega_2$ is Lipschitz with Lipschitz constant $2^{\ell(f)+1}$. 
Hence the action $F\curvearrowright\Omega_2$ extends uniquely to an action of $F$ on the completion 
of $\Omega_2$ w.r.t. metric $d_\tau$. However, this completion is not compact. 

One may also try to equip  $\Omega_2$ with the Vietoris topology, 
or endow $X_2$ with a metric and then put the corresponding Hausdorff distance on $\Omega_2$, 
\cite{Eng}. However, the action of $F$ on $\Omega_2$ is not continuous w.r.t. the Vietoris topology. 
\hfill$\Box$
\end{remark}

For an $\omega\in\Sigma_2$ and a $k\in\Nb$ we denote by $\omega|_k$ the level $k$ restriction of $\omega$. That is, 
$\omega|_k$ is a finite tree such that $\alpha\in\omega|_k$ if and only if $\alpha\in\omega$ and $|\alpha|\leq k$. Also, 
we denote $(\omega,\eta)|_k=(\omega|_k,\eta|_k)$. 

\begin{proposition}\label{boundary}
The following hold true. 
\begin{enumerate}
\item $\Omega_2$ is a discrete and open subspace of $\om$. 
\item Its complement $\bm:=\om\setminus\Omega_2$ is contained in $\Sigma_2\times\Sigma_2$ and homeomorphic to the Cantor space. 
\end{enumerate}
\end{proposition}
\begin{proof}
The empty tree $0$ is an isolated point of space $\Sigma_2^*$. Thus the set $\Sigma_2\times\{0\}$ is 
an open and closed neighbourhood of $(1,0)\in\Sigma_2\times\Sigma_2^*$ that separates it from the set
$\{(q,1-q) \mid 1\neq q\in\Omega_2\}$. Hence $1$ is an isolated point in $\om$. Thus for each $f\in F$, point 
$f\cdot 1$ is also isolated in $\om$. Consequently, $\Omega_2$ is discrete and open in $\om$. Thus  
$\bm$ is a closed subspace of $\Sigma_2\times\Sigma_2$. 

To prove that $\bm$ is a Cantor space, it suffices to show the following: for all  $(\omega,\eta)\in\bm$ and all $k\in\Nb$ there exist 
an $m\in\Nb$ and a sequence $(q_n)$ in $\Omega_2$ such that $q_n\neq q_r$ for $n\neq r$, $(q_n,1-q_n)|_k = 
(\omega,\eta)|_k$ and $q_n|_m$ is constant (does not depend on $n$) and different from $\omega|_m$. 
To this end, take an element $(\omega,\eta)\in\bm$ and a $k\in\Nb$. By the very definition of $\bm$, there exists 
a $q_0\in\Omega_2$ such that $(q_0,1-q_0)|_k=(\omega,\eta)|_k$. We may also assume that there exists an $\alpha_0$ 
with $|\alpha_0|>k$ such that for one of its successors, say $\alpha_01$, we have $P_{\alpha_01}\leq q_0$,  
while $P_{\alpha_02}\leq 1-q_0$. Note that if $q'$ is any proper subprojection of $P_{\alpha_01}$ such that 
$q_0-q'\in\Omega_2$ then $q_0|_k = (q_0 - q')|_k$. Now, take a proper subprojection $q_1'$ of $P_{\alpha_01}$ 
such that $q_1:=q_0-q_1'\in\Omega_2\cap\D_2^m$ has the property that $q_1|_m\neq\omega|_m$.   Set $m_1=m$ and 
construct by induction a sequence of mutually disjoint, non-zero, proper subprojections $q_n'$ of $P_{\alpha_01}$ such that 
$q_n:=q_{n-1}-q_n'\in\Omega_2\cap\D_2^{m_n}$, with $m_n>m_{n-1}$, has the property that $q_n|_{m_{n-1}}=
q_{n-1}|_{m_{n-1}}$. Then the sequence $(q_n)_{n=1}^\infty$ of projections in $\Omega_2$ has the required properties. 
\end{proof}

\begin{remark}\label{fibers}\rm 
Let $\pi_1:\Sigma_2\times\Sigma_2\to\Sigma_2$ be the canonical projection onto the first coordinate. It is 
worth noting that the fibers $\pi_1^{-1}(\omega)\cap\bm$ vary dramatically, depending on the choice of $\omega\in\Sigma_2$. 
For example, if $\omega$ is a tree consisting of a single infinite path then $\pi_1^{-1}(\omega)\cap\bm$ has exactly one point, 
namely $(\omega,1)$. On the other hand, $\pi_1^{-1}(1)\cap\bm=\{1\}\times\Sigma_2$.  
\hfill$\Box$
\end{remark}

In closing, we note that element $(1,1)$ of $\bm$ is globally fixed by $F$.


\medskip\noindent
Jeong Hee Hong \\
Department of Mathematics and Computer Science \\
The University of Southern Denmark \\
Campusvej 55, DK--5230 Odense M, Denmark \\
E-mail: hongjh@imada.sdu.dk \\

\smallskip\noindent
Wojciech Szyma{\'n}ski\\
Department of Mathematics and Computer Science \\
The University of Southern Denmark \\
Campusvej 55, DK--5230 Odense M, Denmark \\
E-mail: szymanski@imada.sdu.dk


\begin{thebibliography}{15} 

\bibitem{ABC} V. Aiello, A. Brothier and R. Conti, 
{\it Jones Representations of Thompson’s Group $F$ Arising from Temperley–Lieb–Jones Algebras}, Internat. Math. Res. Notices 
{\bf 15} (2021), 11209--11245.

\bibitem{AJ} V. Aiello and V. F. R. Jones, 
{\it On spectral measures for certain unitary representations of R. Thompson's group $F$}, J. Funct. Anal. {\bf 280} (2021), 108777, 27 pp. 

\bibitem{BB} J. M. Belk and K. S. Brown, 
{\it Forest diagrams for elements of Thompson's group $F$}, 
Internat. J. Algebra Comput. {\bf 15} (2005), 815--850. 

\bibitem{B} J.-C. Birget, 
{\it The groups of Richard Thompson and complexity}, 
International Conference on Semigroups and Groups in honor of the 65th birthday of Prof. John Rhodes,  
Internat. J. Algebra Comput. {\bf 14} (2004), 569--626. 

\bibitem{BS} M. Brin and C. Squier, 
{\it Groups of piecewise linear homeomorphisms of the real line}, Invent. Math. {\bf 79} (1985), 485--498. 

\bibitem{BJa} A. Brothier and V. F. R. Jones, 
{\it Pythagorean representations of Thompson's groups}, J. Funct. Anal. {\bf 277} (2019), 2442--2469. 

\bibitem{BJb} A. Brothier and V. F. R. Jones, 
{\it On the Haagerup and Kazhdan properties of R. Thompson's groups}, J. Group Theory {\bf 22} (2019), 795--807. 


\bibitem{CFP} J. W. Cannon, W. J. Floyd and W. R. Parry,  
{\it Introductory notes on Richard Thompson's groups}, 
Enseign. Math. (2) {\bf 42} (1996), 215--256. 

\bibitem{Cun} J. Cuntz,
{\it Simple $C^*$-algebras generated by isometries},
Commun. Math. Phys. {\bf 57} (1977), 173--185.

\bibitem{Eng} R. Engelking, 
{\it General topology}, 
2$^{\rm nd}$ ed. Sigma Series in Pure Math. {\bf 6}, Heldermann Verlag, Berlin, 1989.

\bibitem{F}  D. S. Farley, 
{\it Proper isometric actions of Thompson's groups on Hilbert space}, 
Internat. Math. Res. Not. {\bf 45} (2003), 2409--2414.

\bibitem{F08} D. Farley, 
{\it The action of Thompson's group on a CAT(0)  boundary,} 
Groups Geom. Dyn. {\bf  2}  (2008), 185--222. 

\bibitem{Garn} {\L}. Garncarek, 
{\it Analogs of principal series representations for Thompson's groups $F$  and $T$}, 
Indiana Univ. Math. J. {\bf 61}  (2012),  619--626. 

\bibitem{GS} G. Golan and M. Sapir, 
{\it  On the stabilizers of finite sets of numbers in the R. Thompson group $F$},  
Algebra i Analiz {\bf 29}  (2017),  70--110.  

\bibitem{HO} U. Haagerup and K. K. Olesen, 
{\it Non-inner amenability of the Thompson groups $T$  and $V$},  
J. Funct. Anal. {\bf 272}  (2017),  4838--4852. 

\bibitem{HP}  U. Haagerup and G. Picioroaga, 
{\it New presentations of Thompson's groups and applications}, 
J. Operator Theory {\bf 66} (2011), 217--232. 

\bibitem{Hig}  G. Higman, 
{\it Finitely presented infinite simple groups}, 
Notes on Pure Math. {\bf 8}, Australian National Univ., Canberra, 1974. 

\bibitem{J} V. F. R. Jones, 
{\it Some unitary representations of Thompson's groups $F$ and $T$}, 
J. Comb. Algebra  {\bf 1}  (2017),  1--44. 

\bibitem{Mil} K. R. Milliken, 
{\it A partition theorem for the infinite subtrees of a tree},  
Trans. Amer. Math. Soc. {\bf 263}  (1981), 137--148. 

\bibitem{N} V. Nekrashevych, 
{\it Cuntz-Pimsner algebras of group actions}, 
J. Operator Theory {\bf 52} (2004), 223--249. 

\end{thebibliography}
\end{document}